\documentclass[12pt]{article}
\usepackage{amssymb}	
\usepackage{amsmath}    
\usepackage{epsfig}
\usepackage{latexsym}
\usepackage{multirow}
\usepackage{longtable}

\newenvironment{proof}[1][Proof:]{\begin{trivlist} 
\item[\hskip \labelsep {\bfseries #1}]}{\end{trivlist}} 
\newcommand{\qed}{\nobreak \ifvmode \relax \else \ifdim\lastskip<1.5em \hskip-\lastskip \hskip1.5em plus0em minus0.5em \fi \nobreak \vrule height0.75em width0.5em depth0.25em\fi} 
\def\0{\bf \0}
\def\A{{\bf A}}

\def\D{{\bf D}}

\def\I{{\bf I}}

\def\0{{\bf 0}}

\def\R{\mathbb{R}}
\def\S{{\bf S}}
\def\T{{\bf T}}

\def\X{{\bf X}}

\def\a{{\bf a}}
\def\b{{\bf b}}
\def\c{{\bf c}}

\def\e{{\bf e}}

\def\r{{\bf r}}
\def\s{{\bf s}}

\def\u{{\bf u}}
\def\v{{\bf v}}

\def\x{{\bf x}}
\def\y{{\bf y}}

\def\Tr{{\rm T}}
\def\T{{\rm T}}

\newtheorem{algorithm}{Algorithm}[section]

\newtheorem{theorem}{Theorem}[section]
\newtheorem{lemma}{Lemma}[section]

\newtheorem{remark}{Remark}[section]

\usepackage{color}
\usepackage{ulem}

\begin{document}
\title{An ${\cal O}(nL)$ Infeasible-Interior-Point Algorithm for Linear Programming
}
\author{Yuagang Yang\thanks{US NRC, Office of Research, 21 Church 
Street, Rockville, 20850. Email: yaguang.yang@verizon.net.} 
 and Makoto Yamashita\thanks{
Department of Mathematical and Computing Sciences, Tokyo Institute of Technology. Email: Makoto.Yamashita@is.titech.ac.jp.}}

\date{\today}

\maketitle    

\begin{abstract}
In this paper, we propose an arc-search infeasible-interior-point
algorithm. We show that this algorithm is polynomial and the 
polynomial bound is ${\cal O}(nL)$ which is at least as good as
the best existing bound for infeasible-interior-point algorithms
for linear programming.
\end{abstract}

{\bf Keywords:} polynomial algorithm, infeasible-interior-point method, 
linear programming.

\newpage
 
\section{Introduction}

Since Klee and Minty \cite{kleeMinty72} showed that a simplex method
for linear programming is not a polynomial algorithm, polynomial complexity bound has become a popular metric to measure the
efficiency of optimization algorithms.
Searching for polynomial algorithms for linear programming was a major
research area of optimization between 1980's and 1990's after
Khachiyan \cite{Khachiyan79} announced the first polynomial algorithm for 
linear programming. Although Khachiyan's algorithm was  shown to
be much less efficient in practice than the simplex method \cite{bland81}, Karmarkar's
interior-point method \cite{karmarkar84} demonstrated the possibility
of existence of efficient polynomial algorithms. For feasible starting point, 
people quickly established polynomial bounds for various interior-point 
algorithms \cite{Renegar88,kmy89a,kmy89b,Gonzaga90,ye91,mty93}. The lowest 
bound of these algorithms is ${\cal O}(\sqrt{n}L)$ which 
has not been improved for more than two decades. 

To have an efficient implementation for interior-point algorithms, 
Mehrotra \cite{Mehrotra92} and Lustig et. al. \cite{lms92} 
realized that higher-order method and infeasible starting point are two necessary
improvements. However, algorithms with either one of these features
had poorer complexity bounds than ${\cal O}(\sqrt{n}L)$.
Monteiro, Adler, and Resende \cite{mar90} showed that a 
higher-order algorithm starting from a feasible point has the polynomial bound ${\cal O}({n}L)$. 
For infeasible-interior-point method, Zhang \cite{zhang94},
Mizuno \cite{mizuno94}, and Miao \cite{miao96} established 
polynomiality for several different algorithms (none of them is a
higher-order algorithm). The best 
complexity bound ${\cal O}({n}L)$ for infeasible interior-point methods 
has not been changed since the eary of 1990's.

Recently, Yang \cite{yang13,yang11} showed that for a 
higher-order interior-point method 
starting from a feasible point, the polynomial bound can be improved to
${\cal O}(\sqrt{n}L)$ by using an arc-search method. 
Very recently, Yang et. al. \cite{yzl15} used the same idea and 
proposed a polynomial arc-search infeasible-interior-point 
algorihtm with a complexity bound of 
${\cal O}(n^{\frac{5}{4}}L)$.
In this paper, we show that for higher-order infeasible-interior-point method using arc-search, the polynomial bound can be 
improved to ${\cal O}(nL)$,
which is a bound at least as good as the best bound of existing infeasible-interior-point algorithms.

The remainder of the paper is organized as follows. Section 2 describes 
the problem. Section 3 provides an infeasible-predictor-corrector algorithm. 
Section 4 proves its polynomiality. Section 5 summarizes the
conclusions.

\section{Problem Descriptions}

The standard form of linear programming in this paper is given as follows:
\begin{eqnarray}
\min \hspace{0.05in} \c^{\T}\x, \hspace{0.15in} \mbox{\rm subject to} 
\hspace{0.1in}  \A\x=\b, \hspace{0.1in} \x \ge 0,
\label{LP}
\end{eqnarray}
where $\A \in {\R}^{m \times n}$, $\b \in {\R}^{m} $, $\c \in {\R}^{n}$ 
are given, and $\x \in {\R}^n$  is the vector to be optimized. Associated 
with the linear programming is the dual programming that is also presented in the standard form:
\begin{eqnarray}
\max \hspace{0.05in} \b^{\T}\y, \hspace{0.15in} \mbox{\rm subject to} 
\hspace{0.1in}  \A^{\T}\y+\s=\c, \hspace{0.1in} \s \ge 0,
\label{DP}
\end{eqnarray}
where dual variable vector $\y \in {\R}^{m}$, and dual slack vector 
$\s \in {\R}^{n}$. We use ${\cal S}$ to denote the set of all the optimal solutions $(\x^*,\y^*,\s^*)$ 
of (\ref{LP}) and
(\ref{DP}). It is well known that $\x \in {\R}^n$ is an optimal solution 
of (\ref{LP}) if and only if $\x$, $\y$, and $\s$ satisfy the 
following KKT conditions
\begin{subequations}
\begin{align}
\A\x=\b,  \\
\A^{\T}\y+\s=\c, \\
(\x,\s) \ge 0, \\
x_is_i = 0, \hspace{0.1in} i=1,\ldots,n.
\end{align}
\label{ifonlyif}
\end{subequations}
To simplify the notation, we will denote Hadamard (element-wise) product of two 
vectors $\x$ and $\s$ by $\x \circ \s$, 
the element-wise division of the two vectors by $\s^{-1}\circ \x$, or $\x\circ \s^{-1}$,
or $\frac{{{\x}}}{{\s}}$ if $\min | s_i | >0$, the Euclidean norm of $x$ by $\| \x \|$, 
the infinite norm of $\x$ by $\| \x \|_{\infty}$, 
the identity matrix of any dimension by $\I$, the vector of all ones with appropriate 
dimension by $\e$, 
the block column vectors, for example, $[\x^{\T}, \s^{\T}]^{\T}$ 
by $(\x, \s)$. For $\x \in \R^n$, we will denote a related diagonal matrix 
by $\X \in \R^{n \times n}$ whose diagonal elements are the components
of the vector $\x$. Finally, we define an initial vector point of a sequence 
by $\x^0$, an initial scalar point of a sequence 
by $\mu_0$, the vector point after the $k$th iteration by $\x^k$, the scalar point 
after the $k$th iteration by $\mu_k$. Let 
\begin{subequations}
\begin{align}
\r_b^k=\A\x^k -\b, \\  \r_c^k=\A^{\Tr} \y^k + \s^k-\c.
\end{align}
\label{resi}
\end{subequations}
Given a strictly positive current point $(\x^k, \s^k)>0$, the infeasible-predictor-corrector 
algorithm is to find the solution of (\ref{LP}) approximately along a curve ${\cal C}(t)$
defined by the following system
\begin{subequations}
\begin{align}
\A\x(t)-\b=t \r_b^k  \\
\A^{\T}\y(t) +\s(t)-c=t \r_c^k \\
\x(t) \circ \s(t) = t \x^k \circ \s^k \\
(\x(t),\s(t)) >0, 
\end{align}
\label{curve}
\end{subequations}
where $t \in (0,1]$. As $t \rightarrow 0$, $(\x(t), \y(t), \s(t))$ 
appraches the solution of (\ref{LP}). Since ${\cal C}(t)$ is not easy to obtain,
we will use an ellipse ${\cal E}$ \cite{carmo76} in the $2n+m$ dimensional 
space to approximate the curve defined by (\ref{curve}), where 
${\cal E}$ is given by
\begin{equation}
{\cal E}=\lbrace (\x(\alpha), \y(\alpha), \s(\alpha)): 
(\x(\alpha), \y(\alpha), \s(\alpha))=
\vec{\a}\cos(\alpha)+\vec{\b}\sin(\alpha)+\vec{\c} \rbrace,
\label{ellipse}
\end{equation}
$\vec{\a} \in \R^{2n+m}$ and $\vec{\b} \in \R^{2n+m}$ are the axes of the 
ellipse, and $\vec{\c} \in \R^{2n+m}$ is the center of the ellipse. Taking the derivatives of
(\ref{curve}) gives
\begin{equation}
\left[
\begin{array}{ccc}
\A & \0 & \0\\
\0 & \A^{\T} & \I \\
\S^k & \0 & \X^k
\end{array}
\right]
\left[
\begin{array}{c}
\dot{{\x}} \\ \dot{\y}  \\  \dot{{\s}}
\end{array}
\right]
=\left[
\begin{array}{c}
\r_b^k \\ \r_c^k \\ {\x^k} \circ {\s^k} 
\end{array}
\right],
\label{doty}
\end{equation}

\begin{equation}
\left[
\begin{array}{ccc}
\A & \0 & \0\\
\0 & \A^{\T} & \I \\
\S^k & \0 & \X^k
\end{array}
\right]
\left[
\begin{array}{c}
\ddot{\x} \\ \ddot{\y}  \\  \ddot{\s}
\end{array}
\right]
=\left[
\begin{array}{c}
\0 \\ \0 \\ -2\dot{\x} \circ \dot{\s}
\end{array}
\right].
\label{ddoty}
\end{equation}
We require the ellipse to pass the same point $({\x^k}, \y^k, {\s^k})$ on ${\cal C}(t)$
and to have the same derivatives given by (\ref{doty}) and (\ref{ddoty}).
The ellipse is given in \cite{yang13,yang11} as
\begin{theorem}
Let $(\x(\alpha),\y(\alpha),\s(\alpha))$ be an arc defined by 
(\ref{ellipse}) passing
through a point $(\x,\y,\s)\in {\cal E} \cap {\cal C}(t)$, and its first and second 
derivatives at $(\x,\y,\s)$ be $(\dot{\x}, \dot{\y}, \dot{\s})$ and 
$(\ddot{\x}, \ddot{\y}, \ddot{\s})$ which are defined by
(\ref{doty}) and (\ref{ddoty}). Then, the ellipse approximation of $(\ref{curve})$ is given by
\begin{equation}
\x(\alpha) = \x - \dot{\x}\sin(\alpha)+\ddot{\x}(1-\cos(\alpha)).
\end{equation}
\begin{equation}
\y(\alpha) = \y-\dot{\y}\sin(\alpha)
+\ddot{\y} (1-\cos(\alpha)).
\end{equation}
\begin{equation}
\s(\alpha) = \s - \dot{\s}\sin(\alpha)+\ddot{\s} (1-\cos(\alpha)).
\end{equation}
\label{ellipseSX}
\end{theorem}

\section{Infeasible predictor-corrector algorithm}

We denote the duality measure by
\begin{equation}
\mu = \frac{ \x^{\Tr} \s}{n},
\end{equation}
and define the set of neighborhood by
\begin{equation}
{\cal N}(\theta) := \{ (\x,\s) \,\,\, | \,\,\, (\x,\s)>0, \hspace{0.1in} \| \x \circ \s - \mu \e \| \le \theta \mu \}.
\end{equation}
The proposed algorithm searches an optimizer along the ellipse while staying inside ${\cal N}(\theta)$.

\begin{algorithm} {\ } \\
Data: $\A$, $\b$, $\c$, $\theta \in (0, \frac{1}{2+\sqrt{2}}]$, $\epsilon>0$, initial point 
$({\x}^0, {\y}^0, {\s}^0) \in {\cal N}(\theta)$.
\newline
{\bf for} iteration $k=1,2,\ldots$
\begin{itemize}
\item[] Step 1: If 
\begin{subequations}
\begin{align} \label{11a}
{\mu}_k \le {\epsilon}, \\
\label{11b}
\| \r_b^k \| = \|\A{\x}^k-\b \| \le \epsilon, \\
\label{11c}
\| \r_c^k \| = \|\A^{\T} {\y}^k+{\s}^k-\c \| \le \epsilon, \\
\label{11d}
 ({\x}^k, {\s}^k)>0.
\end{align}
\label{errorNorms}
\end{subequations}
stop. Otherwise continue.
\item[] Step 2: Solve the linear systems of equations (\ref{doty}) 
and (\ref{ddoty}) to get $(\dot{\x}, \dot{\y}, \dot{\s})$ and
$(\ddot{\x}, \ddot{\y}, \ddot{\s})$. 
\item[] Step 3: Find the smallest positive $\bar{\alpha} \in (0, \pi/2]$
such that for all $\alpha \in (0, \bar{\alpha}]$, $(\x(\alpha), \s(\alpha))>\0$ and 
\begin{equation}
\| (\x(\alpha) \circ \s(\alpha)) -(1-\sin(\alpha)) \mu_k \e \| \le 2 \theta (1-\sin(\alpha)) \mu_k.
\label{accurate}
\end{equation}
Set (to simplify the notation, we use $\alpha$ in stead of $\bar{\alpha}$ in the rest of the paper)
\begin{equation}
(\x(\alpha), \y(\alpha), \s(\alpha))=({\x}^k,{\y}^k,{\s}^k)
-(\dot{\x}, \dot{\y}, \dot{\s})\sin(\alpha)
+(\ddot{\x}, \ddot{\y}, \ddot{\s})(1-\cos(\alpha)).
\label{poly2}
\end{equation}
\item[] Step 4: Calculate $(\Delta \x,\Delta \y,\Delta \s)$ by solving
\begin{equation}
\left[
\begin{array}{ccc}
\A & \0 & \0\\
\0 & \A^{\T} & \I \\
\S(\alpha) & \0 & \X(\alpha)
\end{array}
\right]
\left[
\begin{array}{c}
\Delta \x \\ \Delta \y  \\  \Delta \s
\end{array}
\right]
=\left[
\begin{array}{c}
\0 \\ \0 \\ (1-\sin(\alpha)){\mu}_k \e -\x(\alpha) \circ \s(\alpha)
\end{array}
\right].
\label{newtondir1}
\end{equation}
Update
\begin{equation}
({\x}^{k+1}, {\y}^{k+1}, {\s}^{k+1})=(\x(\alpha), \y(\alpha), \s(\alpha))
+(\Delta \x,\Delta \y,\Delta \s)
\label{poly3}
\end{equation}
and 
\begin{equation}
{\mu}_{k+1}=\frac{{\x}^{{k+1}^{\T}}{\s}^{k+1}}{n}.
\label{poly4}
\end{equation}
\item[] Step 5: Set $k+1 \rightarrow k$. Go back to Step 1.
\end{itemize}
{\bf end (for)} 
\hfill \qed
\label{mainAlgo3}
\end{algorithm}

In the rest of this section, we will show (1) $\r_b^k \rightarrow 0$, $\r_c^{k} \rightarrow 0$, 
and $\mu_k \rightarrow 0$; (2) there exist $\alpha \in (0, \pi/2]$ such that 
$(\x(\alpha), \s(\alpha))>0$ and (\ref{accurate}) holds; (3) $(\x^k, \s^k) \in {\cal N}(\theta)$. 
It is easy to show that $\r_b^k$, $\r_c^k$, and $\mu_k$ decrease at the same rate
in every iteration.
\begin{lemma}
\begin{eqnarray}
\r_b^{k+1}=\r_b^k (1-\sin(\alpha)), \hspace{0.1in} \r_c^{k+1}=\r_c^{k} (1-\sin(\alpha)), 
\hspace{0.1in} \mu_{k+1}=\mu_k (1-\sin(\alpha)).
\end{eqnarray}
\label{rate}
\end{lemma}
\begin{proof} Using (\ref{resi}), (\ref{poly3}), (\ref{poly2}), (\ref{ddoty}), and (\ref{doty}), we have
\begin{eqnarray}
& & \r_b^{k+1}-\r_b^k = \A (\x^{k+1}-\x^k) = \A(\x(\alpha) + \Delta \x-\x^k) \nonumber \\
& = & \A(\x^k-\dot{\x} \sin(\alpha)-\x^k) = -\A \dot{\x} \sin(\alpha) =-\r_b^k \sin(\alpha). \nonumber
\end{eqnarray}
This shows the first relation. The second relation follows a similar derivation.
From (\ref{newtondir1}), it holds that 
$(\Delta \x)^{\Tr} \Delta \s = (\Delta \x)^{\Tr} (-\A^{\Tr}\Delta \y) = 
-(\A \Delta \x)^{\Tr} \Delta \y = 0$.
Using (\ref{poly3}), we have
\begin{eqnarray}
& & \x^{{k+1}^{\Tr}} \s^{k+1}=(\x(\alpha)+\Delta \x)^{\Tr} (\s(\alpha)+\Delta \s)
=\x(\alpha)^{\Tr}  \s(\alpha)+\x(\alpha)^{\Tr}\Delta \s + \s(\alpha)^{\Tr}\Delta \x  \nonumber \\
& = & \x(\alpha)^{\Tr}  \s(\alpha) +(1-\sin(\alpha)) \mu_k n - \x(\alpha)^{\Tr}  \s(\alpha)
=(1-\sin(\alpha)) \mu_k n. \nonumber 
\end{eqnarray}
Dividing both sides by $n$ proves the last relation. 
\hfill \qed
\end{proof}

Clearly, if $\sin(\alpha)=1$ ($\alpha = \frac{\pi}{2}$), we will find the optimal solution 
(allowing some $x_i=0$ and/or $s_j=0$) in one step, which is rarely the case. Therefore, 
from now on, we assume $\alpha \in (0, \frac{\pi}{2})$. We will use the following lemma of \cite{miao96}.

\begin{lemma} Let $(\Delta \x, \Delta \s)$ be given by (\ref{newtondir1}). Then
\begin{equation}
\| \Delta \x \circ \Delta \s \| \le \frac{\sqrt{2}}{4} 
\|  (\X(\alpha)\S(\alpha))^{-\frac{1}{2}}(\x(\alpha) \circ \s(\alpha)-\mu_{k+1} \e) \|^2.
\end{equation}
\label{mizuno}
\end{lemma}

\begin{theorem} Let $(\x(\alpha), \y(\alpha), \s(\alpha))$ and
$(\x^{k},\s^{k}) \in {\cal N}(\theta)$. Then, for all $k \ge 0$
\begin{itemize}
\item [(i)] there is an $\alpha > 0$, such that $(\x(\alpha), \s(\alpha))>0$ and (\ref{accurate})
holds.
\item [(ii)] if $\theta \le \frac{1}{2+\sqrt{2}}$, then 
$(\x^{k+1},\s^{k+1}) \in {\cal N}(\theta)$ for all the iterations.
\end{itemize}
\label{neighbor}
\end{theorem}
\begin{proof} Using $1-\cos(\alpha) \le 1-\cos^2(\alpha) =\sin^2 (\alpha)$ and 
$({\x}^k, {\y}^k, {\s}^k) \in {\cal N}(\theta)$, we have
\begin{eqnarray}
& & \| \x(\alpha) \circ \s(\alpha) - (1-\sin(\alpha)) \mu_k \e \|  \nonumber \\
& = & \| (\x^k \circ \s^k - \mu_k \e)(1-\sin(\alpha)) +(\ddot{\x} \circ \ddot{\s}-\dot{\x} \circ \dot{\s})(1-\cos (\alpha))^2
 \nonumber \\
& & -(\ddot{\x} \circ \dot{\s} + \dot{\x} \circ \ddot{\s})\sin(\alpha) (1-\cos(\alpha)) \|
\nonumber \\
& \le & \theta \mu_k (1-\sin(\alpha))+( \| \ddot{\x} \circ \ddot{\s}\| + \| \dot{\x} \circ \dot{\s} \| )\sin^4(\alpha)
\nonumber \\
& & + ( \| \ddot{\x} \circ \dot{\s} \| + \| \dot{\x} \circ \ddot{\s} \| ) \sin^3(\alpha).
\end{eqnarray}
Clearly, if 
\begin{equation}
q(\alpha) :=
\left(\Bigl\lVert  \ddot{\x} \circ \ddot{\s} \Bigr\rVert
+\Bigl\lVert  \dot{\x} \circ \dot{\s} \Bigr\rVert \right)\sin^4(\alpha)
+\left(\Bigl\lVert  \dot{\x} \circ \ddot{\s} \Bigr\rVert
+\Bigl\lVert  \ddot{\x} \circ \dot{\s} \Bigr\rVert \right)\sin^3(\alpha)
+\theta {\mu}_k\sin(\alpha)-\theta {\mu}_k \le 0,
\label{alpha1}
\end{equation}
then, (\ref{accurate}) holds. Indeed, since $q(0) = -\theta \mu < 0$, by continuity, 
there exist $\alpha >0$ such that (\ref{alpha1}) holds. This shows that (\ref{accurate}) holds. 
From (\ref{accurate}), we have 
\[
x_i(\alpha)s_i(\alpha) \ge (1-2\theta)(1-\sin(\alpha)) \mu_{k}>0, \hspace{0.1in} \forall \theta \in [0,0.5) 
\hspace{0.1in} \mbox{and} \hspace{0.1in} \forall \alpha \in [0, \pi/2).
\]
This shows $(\x(\alpha), \s(\alpha))>0$. Therefore, we finish part (i). 
Furthermore, from Lemma 3.1, (15) is now equivalent to 
$\| \x(\alpha) \circ \s(\alpha)- \mu_{k+1} \e \| \le 2 \theta \mu_{k+1}$.
Using (\ref{poly3}), 
(\ref{newtondir1}), Lemmas \ref{rate} and \ref{mizuno}, and part (i) of this theorem, we have
\begin{eqnarray}
& & \| \x^{k+1} \circ \s^{k+1} - \mu_{k+1} \e \| \nonumber \\
& = & \| (\x(\alpha) +\Delta \x) \circ (\s(\alpha) +\Delta \s) - \mu_{k+1} \e \|   \nonumber \\
& = & \| \Delta \x \circ \Delta \s \|  \le \frac{\sqrt{2}}{4} 
\|  (\X(\alpha)\S(\alpha))^{-\frac{1}{2}}(\x(\alpha) \circ \s(\alpha)-\mu_{k+1} \e ) \|^2   \nonumber \\
& \le & \frac{\sqrt{2}}{4} \frac{\| \x(\alpha) \circ \s(\alpha)-\mu_{k+1} \e \|^2 }{\min_i x_i(\alpha)s_i(\alpha)}
   \nonumber \\
& \le &    \frac{\sqrt{2}(2\theta)^2 \mu_{k+1}^2}{4(1-2\theta ) \mu_{k+1}}    \nonumber \\
& \le &   \frac{\sqrt{2}\theta^2}{(1-2\theta )  }\mu_{k+1}.
\label{later}
\end{eqnarray}
It is easy to check that for $\theta \le \frac{1}{2+\sqrt{2}} \approx 0.29289$, 
$\frac{\sqrt{2}\theta^2}{(1-2\theta )  } \le \theta$ holds, therefore,
for $\theta \le \frac{1}{2+\sqrt{2}}$, we have
\[
\| \x^{k+1} \circ \s^{k+1} - \mu_{k+1} \e \| \le \theta \mu_{k+1}.
\]
We now show that $(\x^{k+1}, \s^{k+1})>0$. Let $\x^{k+1}(t)=\x(\alpha)+t \Delta \x$ and
$\s^{k+1}(t)=\s(\alpha)+t \Delta \s$. Then, $\x^{k+1}(0)=\x(\alpha)$ and $\x^{k+1}(1)=\x^{k+1}$.
Since
\begin{eqnarray}
& &  \x^{k+1}(t) \circ \s^{k+1}(t) = (\x(\alpha)+t \Delta \x) \circ (\s(\alpha)+t \Delta \s)
\nonumber \\
& = & \x(\alpha) \circ \s(\alpha) +t (\x(\alpha) \circ \Delta \s+\s(\alpha)  \circ \Delta \x)
+ t^2 \Delta \x \circ \Delta \s, \nonumber
\end{eqnarray}
using (\ref{newtondir1}), (\ref{accurate}), (\ref{later}), and the assumption that 
$\theta \le \frac{1}{2+\sqrt{2}}$, we have
\begin{eqnarray}
& &  \| \x^{k+1}(t) \circ \s^{k+1}(t) - {\mu}_{k+1} \e \| 
\nonumber \\
& = & \| (1-t) ( \x(\alpha) \circ \s(\alpha) - {\mu}_{k+1} \e) 
+ t^2 \Delta \x \circ \Delta \s \| \nonumber \\
& \le & 2(1-t) \theta {\mu}_{k+1} + t^2 \frac{\sqrt{2}\theta^2}{1-2\theta}{\mu}_{k+1}
\nonumber \\
& \le & (2(1-t)+t^2) \theta {\mu}_{k+1}:=f(t) \theta {\mu}_{k+1}.
\end{eqnarray}
The function $f(t)$ is a monotonical decreasing function of $t \in [0,1]$,
and $f(0) = 2$. This proves 
$\| \x^{k+1}(t) \circ \s^{k+1}(t) - {\mu}_{k+1} \e \| \le 2\theta {\mu}_{k+1}$.
Therefore, 
$x_i^{k+1}(t) s_i^{k+1}(t) \ge (1-2\theta) {\mu}_{k+1}>0$ for all $t \in [0,1]$, 
which means $(\x^{k+1} \s^{k+1})>0$. This finishes the proof of part (ii).
\hfill
\qed
\end{proof}

This theorem indicates that the proposed algorithm is well-defined.

\section{Polynomiality}

The analysis follows similar ideas in many existing literatures, such as \cite{miao96,kojima96}.
Let the initial point be selected to satisfy 
\begin{equation}
( \x^0, \s^0) \in {\cal N}(\theta),\hspace{0.1in} \x^* \le \rho \x^0,\hspace{0.1in} \s^* \le \rho \s^0,\hspace{0.1in}
(\x^*, \y^*, \s^*) \in {\cal S},
\label{initial}
\end{equation}
where $\rho \ge 1$. Let $\omega^f$ and $\omega^o$ be the
quality of the initial point which are the ``distances'' from feasibility and optimility 
given by
\begin{equation}
\omega^f = \min_{\x,\y,\s} \{ \max \{
\| (\X^0)^{-1}(\x-\x^0) \|_{\infty}, \| (\S^0)^{-1}(\s-\s^0) \|_{\infty} \} \,\,\,
| \,\,\, \A\x=\b, \,\,\, \A^{\Tr}\y +\s=\c \}.
\end{equation}
and 
\begin{equation}
\omega^0 = \min_{\x^*,\y^*,\s^*} \{ \max \{
\frac{\x^{*^{\Tr}}\s^0}{\x^{0^{\Tr}}\s^0}, \frac{\s^{*^{\Tr}}\x^0}{\x^{0^{\Tr}}\s^0}, 1 \} \,\,\,
| \,\,\, (\x^*,\y^*,\s^*) \in {\cal S} \}.
\end{equation}
Let $\omega^r_p$ and $\omega^r_d$ be the ``ratios'' of the feasibility and the 
total complementarity defined by
\begin{subequations}
\begin{align}
\omega^r_p = \frac{\| \A\x^{0}-\b \|}{\x^{0^{\Tr}}\s^0},
\\
\omega^r_d = \frac{\| \A^{\Tr} \y^0 +\s^0-\c \|} {\x^{0^{\Tr}}\s^0}.
\end{align}
\end{subequations}
In view of Lemma \ref{rate}, we have that 
\begin{subequations}
\begin{align}
\| \A\x^k - \b \| = \omega^r_p  {\x^{k^{\Tr}}\s^k},
\\
\| \A^{\Tr} \y^k +\s^k-\c \| = \omega^r_d {\x^{k^{\Tr}}\s^k}.
\end{align}
\end{subequations}
Invoking Lemma 3.3 of \cite{kojima96} for $\lambda_p=\lambda_d=\xi=1$ and (\ref{doty}), 
we have the following two Lemmas \cite{miao96}.
\begin{lemma}
Let $(\dot{\x}, \dot{\s})$ be defined by (\ref{doty}), and 
$\D^k = (\X^k)^{\frac{1}{2}}(\S^k)^{-\frac{1}{2}}$. 
Then
\begin{equation}
\max \{ \| (\D^k)^{-1} \dot{\x} \|, \| (\D^k) \dot{\s} \| \}
\le \| (\x^k \circ \s^k)^{\frac{1}{2}} \| + \omega^f (1+2\omega^o) 
\frac{(\x^k)^{\Tr} \s^k}{\min_i (x_i^k s_i^k)^{\frac{1}{2}}}.
\label{kojima}
\end{equation} 
\label{kojima1}
\end{lemma}
\begin{lemma}
Let $({\x}^0, {\s}^0)$ be defined by (\ref{initial}). Then
\begin{equation}
\omega^f \le \rho, \hspace{0.1in} \omega^o \le \rho.
\label{kojima2f}
\end{equation} 
\label{kojima2}
\end{lemma}
This leads to the following lemma.
\begin{lemma}
Let $(\dot{\x}, \dot{\s})$ be defined by (\ref{doty}). Then, there exists
a positive constant $C_0$, independent of $n$, such that
\begin{equation}
\max \{ \| (\D^k)^{-1} \dot{\x} \|, \| (\D^k) \dot{\s} \| \}
\le C_0 \sqrt{n (\x^k)^{\Tr} \s^k}.
\label{estimateXS}
\end{equation}
\label{estimateDotXS}
\end{lemma}
\begin{proof}
First, it is easy to see 
\begin{equation}
\| (\x^k \circ \s^k)^{\frac{1}{2}} \| = 
\sqrt{\sum_i x_i^k s_i^k} = \sqrt{(\x^k)^{\Tr} \s^k}. 
\label{est1}
\end{equation}
Since $(\x^k, \s^k) \in {\cal N}(\theta)$, we have 
$\min_i (x_i^k s_i^k) \ge (1-\theta) \mu_k = (1-\theta) \frac{ (\x^k)^{\Tr} \s^k}{n}$.
Therefore,
\begin{equation}
\frac{(\x^k)^{\Tr} \s^k}{\min_i (x_i^k s_i^k)^{\frac{1}{2}}}
\le \sqrt{ \frac{n(\x^k)^{\Tr} \s^k}{(1-\theta)}}.
\label{est2}
\end{equation}
Substituting (\ref{est1}) and (\ref{est2}) into (\ref{kojima}) and using Lemma \ref{kojima2} prove (\ref{estimateXS}) with 
$C_0=1+\frac{\rho (1+2\rho)}{\sqrt{(1-\theta)}} \ge 1+\frac{\omega^f (1+2\omega^o)}{\sqrt{(1-\theta)}}$.
\hfill \qed
\end{proof}

From Lemma \ref{estimateDotXS}, we can establish several useful inequalities. The following simple 
facts will be used several times. Let $\u$ and $\v$ be two vectors, then
\begin{equation}
\| \u \circ \v \|^2 = \sum_i (u_i v_i)^2 \le \left(\sum_i u_i^2 \right) \left( \sum_i v_i^2 \right) = \| \u \|^2 \| \v \|^2.
\label{fact1}
\end{equation}
If $\u$ and $\v$ satisfy $\u^{\Tr} \v =0$, then,
\begin{equation}
\max \{ \| \u \|^2, \| \v \|^2 \} \le \| \u \|^2+\| \v \|^2 = \| \u + \v \|^2,
\label{fact2}
\end{equation}
and (see \cite[Lemma 5.3]{wright97})
\begin{equation}
\| \u \circ \v \| \le 2^{-\frac{3}{2}} \| \u + \v \|^2.
\label{fact3}
\end{equation}

\begin{lemma}
Let $(\dot{\x}, \dot{\s})$ and $(\ddot{\x}, \ddot{\s})$ be defined by (\ref{doty}) and  (\ref{ddoty}), respectively.  
Then, there exist positive constants $C_1$, $C_2$, $C_3$, and $C_4$, independent of $n$, such that
\begin{eqnarray}
\| \dot{\x} \circ \dot{\s} \| \le C_1 n^2 \mu_k, \label{est3} \\
\| \ddot{\x} \circ \ddot{\s} \| \le C_2 n^4 \mu_k, \label{est4}  \\
\max \{ \| (\D^k)^{-1} \ddot{\x} \|, \| (\D^k) \ddot{\s} \|  \} \le C_3 n^{2} \sqrt{\mu_k}, \label{est5}  \\
\max \{ \| \ddot{\x} \circ \dot{\s} \|, \| \dot{\x} \circ \ddot{\s} \| \} \le C_4n^{3}\mu_k \label{est6} 
\end{eqnarray}
\label{norms}
\end{lemma}
\begin{proof}
First, using (\ref{fact1}) and Lemma \ref{estimateDotXS}, we have
\begin{eqnarray}
\| \dot{\x} \circ \dot{\s} \| = \| (\D^k)^{-1} \dot{\x} \circ (\D^k) \dot{\s}  \| \le 
\| (\D^k)^{-1} \dot{\x} \| \| (\D^k) \dot{\s}  \| \le C_0^2 {n (\x^k)^{\Tr} \s^k} := C_1 n^2 \mu_k.
\end{eqnarray}
Second, using (\ref{fact3}), (\ref{ddoty}), (\ref{est3}), and (\ref{est1}), we have 
\begin{eqnarray}
& & \| \ddot{\x} \circ \ddot{\s} \| = \| (\D^k)^{-1} \ddot{\x} \circ (\D^k) \ddot{\s}  \| \le 
2^{-\frac{3}{2}} \| (\D^k)^{-1} \ddot{\x} + (\D^k) \ddot{\s} \|^2 
\nonumber \\
& \le & 2^{-\frac{3}{2}} \Bigl\lVert {-2 (\X\S)^{-\frac{1}{2}}(\dot{\x} \circ \dot{\s}) } \Bigr\rVert^2  
\nonumber \\
&=& 
2^{\frac{1}{2}}  \sum_{i=1}^n \left( \frac{\dot{x}_i \dot{s}_i}{\sqrt{x_i}\sqrt{s_i}} \right)^2 
= 2^{\frac{1}{2}}  \sum_{i=1}^n \frac{(\dot{x}_i \dot{s}_i)^2}{x_is_i}  
\nonumber \\
&\le&  2^{\frac{1}{2}}  \frac{\sum_{i=1}^n (\dot{x}_i \dot{s}_i)^2}{\min_{i=1,\ldots,n}x_is_i}  \nonumber \\
&\le& 2^{\frac{1}{2}}  \frac{||\dot{x}\circ\dot{s}||^2}{(1-\theta)\mu_k} 
\le 2^{\frac{1}{2}} \frac{C_1^2 n^4 \mu_k^2}{(1-\theta)\mu_k} 
\nonumber \\
& = & 2^{\frac{1}{2}}  \frac{C_1^2 n^4 \mu_k}{1-\theta}
:=C_2n^4 \mu_k.
\end{eqnarray}
Third, using (\ref{fact2}), (\ref{ddoty}), and (\ref{est3}), we have 
\begin{eqnarray}
& & \max \{ \| (\D^k)^{-1} \ddot{\x} \|^2, \| (\D^k) \ddot{\s} \|^2 \}
\le \| (\D^k)^{-1} \ddot{\x} + (\D^k) \ddot{\s} \|^2 
\nonumber \\
& = & \Bigl\lVert {-2 (\X\S)^{-\frac{1}{2}}(\dot{\x} \circ \dot{\s}) } \Bigr\rVert^2
\le \frac{4 C_1^2 n^4 \mu_k}{1-\theta} :=C_3^2 n^4 \mu_k.
\end{eqnarray}
Taking square root on both sides proves (\ref{est5}). Finally, using (\ref{fact1}), 
(\ref{est5}), and Lemma
\ref{estimateDotXS}, we have
\begin{eqnarray}
& & \| \ddot{\x} \circ \dot{\s} \|=\| (\D^k)^{-1} \ddot{\x} \circ (\D^k) \dot{\s} \| 
\le \| (\D^k)^{-1} \ddot{\x} \| \| (\D^k) \dot{\s} \| 
\nonumber \\
& \le &  (C_3 n^{2} \sqrt{\mu_k})(C_0 n \sqrt{\mu_k}):= C_4n^{3}\mu_k.
\end{eqnarray}
Similarly, we can show
\begin{eqnarray}
\| \dot{\x} \circ \ddot{\s} \| \le C_4n^{3}\mu_k.
\end{eqnarray}
This finishes the proof.
\hfill \qed
\end{proof}

Now we are ready to estimate a conservative bound for $\sin(\alpha)$. 
\begin{lemma}
Let $(\x^k, \y^k, \s^k)$ be generated by Algorithm \ref{mainAlgo3}. Then, $\sin(\alpha)$
obtained in Step 3 satisfies the following inequality.
\begin{equation}
\sin(\alpha) \ge \frac{\theta}{2Cn},
\end{equation}
where $C=\max \{ 1, C_4^{\frac{1}{3}}, (C_1+C_2)^{\frac{1}{4}} \}$.
\label{lastLemma}
\end{lemma}
\begin{proof}
Let $\sin(\alpha) = \frac{\theta}{2Cn}$. In view of (\ref{alpha1}) and Lemma \ref{norms}, we have
\begin{eqnarray}
q(\alpha)  & \le & \mu_k ((C_1+C_2) n^4 \sin^4(\alpha) + 2C_4n^{3} \sin^3(\alpha) +\theta \sin(\alpha)
-\theta) := \mu_k  p(\alpha)
\nonumber \\
& \le & \mu_k \left(
\frac{(C_1+C_2)\theta^4}{16C^4}+\frac{2C_4 \theta^3}{8C^3}+\frac{\theta^2}{2Cn}-\theta
\right)
\nonumber \\
& \le & \mu_k \left(
\frac{\theta^4}{16}+\frac{ \theta^3}{4}+\frac{\theta^2}{2}-\theta
\right)
\le 0. \nonumber 
\end{eqnarray}
Since $p(\alpha)$ is a monotonic function of $\sin(\alpha)$, 
for all $\sin(\alpha) \le \frac{\theta}{2Cn}$,
the above inequalities hold (the last inequality holds because of $\theta \le 1$). 
Therefore, for all $\sin(\alpha) \le \frac{\theta}{2Cn}$, the
inequality (\ref{accurate}) holds. This finishes the proof.
\hfill \qed
\end{proof}

\begin{remark}
It is worthwhile to point out that the constant $C$ depends on $C_0$ which depends on $\rho$,
but $\rho$ is an unknown before we find the solution. Also, we can always find a better steplength 
$\sin(\alpha)$ by solving the quartic $q(\alpha)=0$ and the calculation of the roots for 
a quartic polynomial is deterministic, negligible, and independent
to $n$ \cite{shmakov11,yang13a}.
\end{remark}

Following the standard argument developed in \cite{wright97}, we have
the main theorem.
\begin{theorem}
Let $(\x^k, \y^k, \s^k)$ be generated by Algorithm \ref{mainAlgo3} with an initial point given
by (\ref{initial}). For any $\epsilon >0$, the algorithm will terminate with $(\x^k, \y^k, \s^k)$
satisfying (\ref{errorNorms}) in at most ${\cal O}(nL)$ iterations, where
\[
L=\max \{ \ln((\x^0)^{\Tr} \s^0/\epsilon), \ln(\| \r_b^0\|/\epsilon), \ln(\| \r_c^0\|/\epsilon) \}.
\]
\end{theorem}
\begin{proof}
In view of Lemma \ref{rate}, $\r_b^k$, $\r_c^k$, and $\mu_k$ decrease at the same rate
$(1-\sin(\alpha))$ in every iteration. Using the Lemma \ref{lastLemma} and \cite[Theorem 3.2]{wright97}
proves the claim.
\hfill \qed
\end{proof}


\section{Conclusions}
We proposed an infeasible-interior-point algorithm that searches the optimizer along an
ellipse that approximates the central-path. We showed that the proposed algorithm is
polynomial and that the polynomial bound is at least as good as the best existing bound for
infeasible-interior-point algorithms for linear programming.


\end{document}